\documentclass[10pt]{amsart}
\usepackage{hyperref}
\usepackage{amsmath}
\usepackage{amssymb}
\usepackage{xypic}
\usepackage{yhmath}
\usepackage{setspace}
\def\beqnn{\begin{eqnarray*}}\def\eeqnn{\end{eqnarray*}}

\newtheorem{theorem}{Theorem}[section]
\newtheorem{lemma}[theorem]{Lemma}
\newtheorem{proposition}[theorem]{Proposition}
\newtheorem{question}[theorem]{Question}

\theoremstyle{definition}

\theoremstyle{remark}
\newtheorem{remark}[theorem]{Remark}

\numberwithin{equation}{section}

\begin{document}

\begin{center}
\title[Sharp multiplier estimates for higher Schwarzians of Koebe function]{Sharp multiplier estimates for the higher-order Schwarzian derivatives of the Koebe function}
\end{center}

\author{Jianjun Jin}
\address{School of Mathematics Sciences, Hefei University of Technology, Xuancheng Campus, Xuancheng 242000, P.R.China}
\email{jin@hfut.edu.cn, jinjjhb@163.com}



\begin{abstract}
In this note we study the multiplier norm estimates for the multiplication operators between weighted Bergman spaces, whose symbols are the higher-order Schwarzian derivatives of univalent functions. We establish sharp multiplier estimates for the higher-order Schwarzian derivatives of the Koebe function. This extends a related result by Shimorin. The proof of our new theorem relies on an explicit formula for the higher-order Schwarzian derivatives of the Koebe function and a recent theorem from our earlier work.  We finally point out that the Koebe function is still the extremal function for certain higher-order Schwarzians of the univalent functions.  
\end{abstract}

\maketitle

\section{{\bf {Introduction}}}
Let $\Delta=\{z:|z|<1\}$ denote the unit disk in the
complex plane $\mathbb{C}$. We use $\mathcal{A}(\Delta)$ to denote the class of all analytic functions in $\Delta$. For $-1<\alpha<\infty$, we define the weighted Bergman space ${\bf A}_{\alpha}^2={\bf A}_{\alpha}^2(\Delta)$ as
$${\bf A}_{\alpha}^2(\Delta)=\{\phi \in \mathcal{A}(\Delta) : \|\phi\|_{\alpha}^2:=(\alpha+1)\iint_{\Delta}|\phi (z)|^2(1-|z|^2)^{\alpha}\frac{dxdy}{\pi}<\infty\}.$$
In his paper \cite{SS}, to study the universal integral means spectrum of univalent functions, Shimorin investigated a multiplication operator from ${\bf A}_{\alpha}^2$ to ${\bf A}_{\alpha+4}^2$, which is induced by the Schwarzian derivative of a univalent function. To recall the results of \cite{SS}, we need to introduce some notations and definitions. 

We let $\mathcal{S}$ be the class of all univalent functions $f$ in $\Delta$ with $f(0)=f'(0)-1=0$.  Let $t\in \mathbb{R}$. For $f\in \mathcal{S}$, the {\em integral means spectrum} $\beta_f(t)$ is defined by
\begin{equation}\label{defi}
\beta_f(t):=\limsup_{r \rightarrow 1^{-}}\frac{\log \int_{-\pi}^{\pi} |f'(re^{i\theta})|^{t}d\theta}{|\log(1-r)|}. 
\end{equation}
The {\em universal integral means spectrum} $B(t)$ is defined by
\begin{equation}
B(t)=\sup\limits_{f\in \mathcal{S}} \beta_f(t).\nonumber
\end{equation}

The {\em multiplication operator} $M_g$ with the symbol $g \in \mathcal{A}(\Delta)$ is defined by
$$M_g(\phi)(z):=g(z)\phi(z), \,\,\phi\in \mathcal{A}(\Delta).$$
For $\gamma\in \mathbb{R}$, we define the growth space $\mathcal{B}_{\gamma}=\mathcal{B}_{\gamma}(\Delta)$ as
$$
\mathcal{B}_{\gamma}(\Delta)=\left\{ \phi \in \mathcal{A}(\Delta) : 
\|\phi\|_{\mathcal{B}_{\gamma}}:=\sup_{z\in \Delta}|\phi(z)|(1-|z|^2)^{\gamma}<\infty \right\}.
$$ 
It is well known, see \cite{zhao}, that 
\begin{proposition}\label{th-1}Let $\beta>\alpha>-1$. 
The multiplication operator $M_g$ is bounded from ${\bf A}_{\alpha}^2$ to ${\bf A}_{\beta}^2$ if and only if $g$ belongs to $\mathcal{B}_{\gamma}$ with $\gamma=(\beta-\alpha)/2.$
\end{proposition}
 
For $\beta>\alpha>-1$, when $M_g$ is bounded from ${\bf A}_{\alpha}^2$ to ${\bf A}_{\beta}^2$, we say $g$ is a {\em multiplier} from ${\bf A}_{\alpha}^2$ to ${\bf A}_{\beta}^2$. The class of all multipliers, denoted by $\mathbf{M}_{\alpha, \beta}$, is a Banach space with the following multiplier norm
$$
\|g\|_{{\bf {M}}_{\alpha, \beta}}:=\sup_{\phi \in {\bf A}^{2}_{\alpha}, \phi \neq 0}\frac{\|g \phi\|_{\beta}}{\|\phi\|_{\alpha}}.
$$ 

For a locally univalent function $f$ in an open domain $\Omega$ of $\mathbb{C}$, the {\em Schwarzian derivative} $S_f$ of $f$ is defined by 
$$S_f(z)=\bigg[\frac{f''(z)}{f'(z)}\bigg]'-\frac{1}{2}\left[\frac{f''(z)}{f'(z)}\right]^2, \,z\in \Omega.$$
Here, $N_f(z):= f''(z)/f'(z)$ is the Pre-Schwarzian derivative of $f$. Shimorin has considered in \cite{SS} the following multiplication operator
$$M_{S_f}(\phi)(z):=S_f(z)\phi(z),\, \phi\in \mathcal{A}(\Delta).$$
It is well known that
 $$|S_f(z)|(1-|z|^2)^2\leq 6,$$
 for all $f\in \mathcal{S}$ so that $S_f$ belongs to $\mathcal{B}_2$ and $\|S_f\|_{\mathcal{B}_{2}}\leq 6,$ see \cite{Du}. Then $S_f$ is a multiplier from ${\bf A}_{\alpha}^2$ to ${\bf A}_{\alpha+4}^2$ for any $f\in \mathcal{S}$. For the multiplier norms of the operator $M_{S_f}$,  Shimorin proved that
\begin{theorem}\label{th-sh}
Let $\alpha>-1$. Then
$$\|S_f\|_{{\bf {M}}_{\alpha, \alpha+4}}\leq 36\frac{\alpha+3}{\alpha+1},$$
holds for all $f\in \mathcal{S}$. 
\end{theorem}
\begin{remark}By applying the above estimates, Shimorin found better bounds for the universal integral means spectrum $B(-2)$. $B(-2)=1$ is the famous Brennan conjecture. Later, using the similar ideas, Hedenmalm and Shimorin obtained the current best upper bound estimate about $B(-2)$ and other universal integral means spectra in the paper \cite{HS-1}. The reader can find more results of the topic about the universal integral means spectrum from \cite{Be, GM, HSo}.
\end{remark}
In a recent work \cite{Do}, Donaire extended Shimorin’s above estimates to the higher-order Schwarzian derivatives of the univalent functions. We next review the results established by Donaire in \cite{Do}. For any $f\in\mathcal{S}$, let $F$ be a function in $\Delta^2$ defined by  
 \[
F(z,w)=\log\frac{f(z)-f(w)}{z-w}\cdot\frac{z}{f(z)}\cdot\frac{w}{f(w)},\,\, {\text{if}}\,\, z\neq w;
\]
and \[
F(z,w)=\log(f'(z))\frac{z^2}{[f(z)]^2},\,\, {\text{if}}\,\, z=w.
\]

Let $p,q \in\mathbb{N}$ and $f\in \mathcal{S}$, we let
\[G_f^{[p,q]}(z,w)=\frac{\partial^{p+q}F}{\partial z^p \partial w^q},\]
and define 
\[S_f^{[p,q]}(z)=G_f^{[p,q]}(z,z).\]
We call $S_f^{[p,q]}$ the {\em higher-order Schwarzian derivatives} of $f$.  It is known that $S_f^{[p,q]}=S_f^{[q,p]}$ and $S_f^{[1,1]}=\frac{1}{6}S_f$. 

The following result was proved by Donaire in \cite{Do}.
\begin{theorem}\label{th-d}
Let $\alpha>-1, p,q\in \mathbb{N}$. Then 
$$\|S_f^{[p,q]}\|_{{\bf {M}}_{\alpha, \alpha+2p+2q}}^2\leq (2p-1)!(2q-1)!\frac{\alpha+2\min\{p,q\}+1}{\alpha+1},$$
holds for all $f\in \mathcal{S}$. 
\end{theorem}
\begin{remark}
We can obtain Theorem \ref{th-sh} when we take $p=q=1$ in Theorem \ref{th-d}.
\end{remark}

On the other hand, the following sharp multiplier norm estimate for the Koebe function also has been obtained in \cite{SS} by Shimorin. There is a close relationship between this sharp multiplier norm estimate and the Brennan conjecture, see \cite[Proposition 8]{SS}. We denote by $\kappa$ the well-known Koebe function, which is defined by
 $$\kappa(z):=\frac{z}{(1-z)^2}, z\in \Delta.$$ 
 \begin{proposition}\label{pro-sh}
 Let $\alpha>-1$. We have
\begin{equation}\|S_{\kappa}\|_{{\bf {M}}_{\alpha, \alpha+4}}=\frac{36(\alpha+3)(\alpha+5)}{(\alpha+2)(\alpha+4)}.\nonumber \end{equation}
 \end{proposition}
 
A natural question is what is the sharp multiplier norm estimates for the higher-order Schwarzian derivatives of the Koebe function. In this note, we will answer this question and establish the sharp multiplier norm estimates for the higher Schwarzians of the Koebe function, which is a generalization of Proposition \ref{pro-sh}.  
 
The rest of the paper is organized as follows. In the next section, we will present the main result of this paper and its proof. The proof relies on an explicit formula for the higher-order Schwarzian derivatives of the Koebe function and a recent theorem from our earlier work.  In Section 3, we present the final remarks, which mainly provide a new proof of Theorem \ref{th-d} and point out that the Koebe function is still the extremal function for certain higher-order Schwarzians of the univalent functions. 
 
\section{{\bf Main result and its proof}}
The main result of this note is the following theorem.
\begin{theorem}\label{th-m}
Let $\alpha>-1, p,q\in \mathbb{N}$. Then 
$$\|S_{\kappa}^{[p,q]}\|_{{\bf {M}}_{\alpha, \alpha+2p+2q}}=\sqrt{{\mathcal{N}}_{\alpha, p, q}}.$$
Where, $${\mathcal{N}}_{\alpha, p, q}=\frac{[(p+q-1)!]^2}{2^{2p+2q}}\frac{\Gamma(\alpha+2+2p+2q)}{\Gamma(\alpha+2)}\Big[\frac{\Gamma(\alpha/2+1)}{\Gamma(\alpha/2+1+p+q)}\Big]^2.$$
\end{theorem}
\begin{remark}Here, $\Gamma$ is the well-known Gamma function, see \cite{AAR}. Proposition \ref{pro-sh} will follow if we take $p=q=1$ in Theorem \ref{th-m}. 
\end{remark}
We will establish an explicit formula for the higher-order Schwarzian derivatives of the Koebe function. Let $n$ be a nonnegative integer and $a$ be a complex number. We define the shifted factorial $(a)_n$ by
  \[(a)_n=a(a+1)\cdots(a+n-1),\]
  for $n>0$, and $(a)_0=1$. We shall use the following Chu-Vandermonde identity for hypergeometric series, see \cite{AAR}.
\begin{lemma}\label{lemma-m}Let $n$ be a nonnegative integer. Let $c$ be a complex number with $(c)_m\neq 0$ for all $m=0,1,\cdots,n$, then we have
\[{}_2F_1\left(\begin{matrix} -n, a \\ c \end{matrix}; 1\right)=\sum_{m=0}^{n}\frac{(-n)_m (a)_m}{m! (c)_m}=\frac{(c-a)_n}{(c)_n}.\]
\end{lemma} 
We now present the key lemma of this paper.
\begin{lemma}\label{key-l}Let $p,q\in \mathbb{N}$ with $p\geq q$. Then \begin{equation}\label{m-lemma-eq}S_{\kappa}^{[p,q]}(z)=-p!\frac{z^{p-q}}{(1-z^2)^{p+q}}\sum_{d=0}^{q-1}C_d z^{2d},\,z\in \Delta.\end{equation}
Here, \begin{equation}\label{m-lemma-eq-1}C_d:=\frac{(p-1)!}{(p-q)!}\frac{(-q)_d(1-q)_d}{d!(p-q+1)_d},\,0\leq d\leq q-1.\end{equation}
\end{lemma} 

\begin{proof}
We let  
 $$K(z,w):=\log\frac{\kappa(z)-\kappa(w)}{z-w}\frac{z}{\kappa(z)}\frac{w}{\kappa(w)}.$$
Then $K(z,w)=\log (1-zw)$ and for $p\geq q$, we first obtain that
$$\frac{\partial^{p}K}{\partial z^p}=-(p-1)!w^p(1-zw)^{-p},$$
so that
\begin{eqnarray}\lefteqn{\frac{\partial^{p+q}K}{\partial z^p \partial w^q}=-(p-1)!\frac{\partial^{p}}{\partial w^q}\Big(w^p(1-zw)^{-p}\Big)}\nonumber \\
&&=-(p-1)!\sum_{a=0}^{q}\binom{q}{a}(w^p)^{(a)}\frac{\partial^{q-a}}{\partial w^{q-a}}[(1-zw)^{-p}].\nonumber
\end{eqnarray}
Note that $$(w^p)^{(a)}=p(p-1)\cdots (p-a+1)w^{p-a}=p(p-1)\cdots (p-a+1)w^{p-a},$$
and 
\[\frac{\partial^{q-a}}{\partial w^{q-a}}[(1-zw)^{-p}]=p(p+1)\cdots (p+q-a-1))z^{q-a}(1-zw)^{-p-q+a}.\]
It follows that 
\begin{eqnarray}S_{\kappa}^{[p,q]}(z)&=&-p!\sum_{a=0}^{q}\binom{q}{a}(p-a+1)\cdots(p+q-a-1)z^{p+q-2a}(1-z^2)^{-p-q+a}\nonumber \\
&=&-p!\sum_{a=0}^{q}\binom{q}{a}\frac{(p+q-a-1)!}{(p-a)!}z^{p+q-2a}(1-z^2)^{-p-q+a}.\nonumber \end{eqnarray}
We write $$T_q(t)=\sum_{a=0}^{q}\binom{q}{a}\frac{(p+q-a-1)!}{(p-a)!}t^{q-a}(1-t)^{a}.$$
Then we see that
\[(1-z^2)^{p+q}S_{\kappa}^{[p,q]}(z)=-p!z^{p-q}T_q(z^2).\]
We will determine the coefficient of each term in $T_q$. Note that
\[(1-t)^{a}=\sum_{b=0}^{a}\binom{a}{b}(-t)^{b},\]
so that $$T_q(t)=\sum_{a=0}^{q}\binom{q}{a}\frac{(p+q-a-1)!}{(p-a)!}t^{q-a}\sum_{b=0}^{a}\binom{a}{b}(-t)^{b}.$$
Consequently, the coefficient $C_d$ of the term $t^d (d\leq q)$ in $T_q$ is
\[C_d=\sum_{a=0}^{q}\binom{q}{a}\frac{(p+q-a-1)!}{(p-a)!}\binom{a}{d-q+a}(-1)^{d-q-a}.\]
Since $0\leq d \leq q$, we see that 
\[C_d=\sum_{a=q-d}^{q}\binom{q}{a}\frac{(p+q-a-1)!}{(p-a)!}\binom{a}{d-q+a}(-1)^{d-q-a}\]
\[=\sum_{j=0}^{d}\binom{q}{q-d+j}\frac{(p+d-j-1)!}{(p-q+d-j)!}\binom{q-d+j}{j}(-1)^{j}\]
Meanwhile, 
\[\binom{q}{q-d+j}=\frac{q!}{(q-d+j)!(d-j)!}, \]

\[\binom{q-d+j}{j}=\frac{(q-d+j)!}{j!(q-d)!},\]

\[(p+d-j-1)!=\frac{(p+d-1)!}{(-1)^{j}(-p-d+1)_j},\]

\[(p-q+d-j)!=\frac{(p-q+d)!}{(-1)^j(-p+q-d)_j}.\]
Then, it follows from the Chu-Vandermonde identity that
\[C_d=\frac{q!(p+d-1)!}{d!(p-d)!(p-q+d)!}\sum_{j=0}^{d}\frac{(-d)_j(-p+q-d)!}{j!(-p-d+1)_j}\]
\[=\frac{q!(p+d-1)!}{d!(p-d)!(p-q+d)!} \frac{(1-q)_d}{(-p-d+1)_d}.\]
Note that $(1-q)_d=0$ when $d\geq q$, we know that the degree of the polynomial $T_q$ is $q-1$. Furthermore, in view of the facts that 
 \[(p+d-1)!=(p-1)!(p)_{d},\,\,(p-q+d)!=(p-q)!(p-q+1)_d,\]
\[(p-d)!=\frac{q!}{(-1)^d(-q)_d},\,\,(-p-d+1)_d=(-1)^d(p)_d,\]
we obtain that  
\begin{equation} C_d=\frac{(p-1)!}{(p-q)!}\frac{(-q)_d(1-q)_d}{d!(p-q+1)_d},\,0\leq d\leq q-1. \nonumber\end{equation}
This shows that all coefficients of $T_q$ are positive and we can write  
\begin{equation}S_{\kappa}^{[p,q]}(z)=-p!\frac{z^{p-q}}{(1-z^2)^{p+q}}\sum_{d=0}^{q-1}C_d z^{2d}. \nonumber\end{equation}
This proves Lemma \ref{key-l}.\end{proof}\begin{remark}
When $p\leq q$, we can similarly obtain that 
\begin{equation}S_{\kappa}^{[p,q]}(z)=-q!\frac{z^{q-p}}{(1-z^2)^{p+q}}\sum_{d=0}^{p-1}\widetilde{C}_d z^{2d}. \nonumber\end{equation}
Here, $$\widetilde{C}_d=\frac{(q-1)!}{(q-p)!}\frac{(-p)_d(1-p)_d}{d!(q-p+1)_d},\,0\leq d\leq p-1.$$
\end{remark}
\begin{lemma}\label{l-2}
Let $\alpha>-1, r\in (0,1), \lambda>0$. Let $\Theta$ be a nonnegative integer. If $2\lambda>\alpha+2$, then  
\begin{equation}\label{a-e-1}
 \|\sqrt{r}z^{\Theta}{(1-rz^2)^{-\lambda}}\|_{\alpha}^2=\frac{\Gamma(\alpha+2)\Gamma(2\lambda-\alpha-2)}{2^{\alpha+1}[\Gamma(\lambda)]^2}\cdot\frac{1+{o}(1)}{(1-r^2)^{2\lambda-\alpha-2}},\, {\textup{as}}\,\, r\rightarrow 1^{-}. 
 \end{equation}
If $0<2\lambda<\alpha+2$, then there is a constant $M>0$ such that for all $r\in (0,1)$,
 \begin{equation}\label{a-e-3}
   \|{(1-rz^2)^{-\lambda}}\|_{\alpha}^2\leq M.
 \end{equation} 
\end{lemma}
\begin{remark}
Here and in the sequel, ${o}(1)$ denotes a quantity depending on $r$ such that ${o}(1)\rightarrow 0$ as $r\rightarrow 1^{-}$, and which may be different in different places. The proof of this lemma is similar as the one of Lemma 2.3 in \cite{Jin-2}.
\end{remark}
\begin{proof}
First, recall that for $\varkappa>0$, 
\begin{equation}
  \frac{1}{(1-z)^{\varkappa}}=\sum_{n=0}^{\infty}\frac{\Gamma(n+\varkappa)}{n!\Gamma(\varkappa)}z^n, z\in \Delta.\nonumber \end{equation}
Then, for $r\in (0,1)$, we have
\begin{equation}\label{l-eq-1}
\frac{\sqrt{r}z^{\Theta}}{(1-rz^2)^{\lambda}}=r^{\frac{\Theta}{2}}\sum_{n=0}^{\infty}\frac{\Gamma(n+\lambda)}{n!\Gamma(\lambda)}r^nz^{2n+\Theta}.
\end{equation}
Note that for any $\phi=\sum_{n=0}^{\infty}a_nz^n\in \mathbf{A}_{\alpha}^2(\Delta)$, we have
\begin{equation}\label{l-eq-2}\|\phi\|_{\alpha}^2=\sum_{n=0}^{\infty}\frac{n!\Gamma(\alpha+2)}{\Gamma(n+\alpha+2)}|a_n|^{2}.
\end{equation}
Hence, it follows from (\ref{l-eq-1}) and (\ref{l-eq-2}) that
  \begin{equation}\label{l-eq-3}
 \|\sqrt{r}z^{\Theta}{(1-rz^2)^{-\lambda}}\|_{\alpha}^2=r^{\Theta}\frac{\Gamma(\alpha+2)}{[\Gamma(\lambda)]^2}\sum_{n=0}^{\infty}\frac{(2n+\Theta)!}{\Gamma(2n+\Theta+\alpha+2)}\Big|\frac{\Gamma(n+\lambda)}{n!}\Big|^2 r^{2n}. 
 \end{equation}
Recall from Stirling's formula that 
 \begin{equation}
 \frac{\Gamma(n+\lambda)}{n!}=(n+1)^{\lambda-1}[1+\widetilde{o}(1)],\,\, {\textup{as}} \,\, n\rightarrow \infty.\nonumber 
 \end{equation}
 Here and later, we use $\widetilde{o}(1)$ to denote a sequence that tends to $0$, i.e., $o(1)\rightarrow 0, \, {\textup{as}}\,\, n \rightarrow \infty$, which may be different in different places. Then it follows from (\ref{l-eq-3}) that 
 \begin{eqnarray}\label{l-eq-5}
\lefteqn{\|\sqrt{r}z^{\Theta}{(1-rz^2)^{-\lambda}}\|_{\alpha}^2}\nonumber \\
&&=r^{\Theta}\frac{\Gamma(\alpha+2)}{[\Gamma(\lambda)]^2}\sum_{n=0}^{\infty}\frac{[1+\widetilde{o}(1)]}{[2n+\Theta+1)]^{\alpha+1}}\cdot{(n+1)^{2\lambda-2}}[1+\widetilde{o}(1)]r^{2n}\nonumber \\
 &&=r^{\Theta}\frac{1}{2^{\alpha+1}}\frac{\Gamma(\alpha+2)}{[\Gamma(\lambda)]^2}\sum_{n=0}^{\infty}{(n+1)^{2\lambda-\alpha-3}}[1+\widetilde{o}(1)]r^{2n}.
 \end{eqnarray}
If $0<2\lambda<\alpha+2$, we know that there is a constant $C>0$ such that 
$$\sum_{n=0}^{\infty}{(n+1)^{2\lambda-\alpha-3}}r^{2n}\leq C,\,\,{\text{for all}}\,\, r\in (0,1).$$
This implies that (\ref{a-e-3}) is true. On the other hand, when $2\lambda>\alpha+2$, we have the following identity which has been shown in \cite[Lemma 3.2]{Jin-2}.  
 \begin{equation}\label{l-eq-10}
\sum_{n=0}^{\infty}{(n+1)^{2\lambda-\alpha-3}}[1+\widetilde{o}(1)]r^{2n}=\frac{1+{o}(1)}{(1-r^2)^{2\lambda-\alpha-2}},\, {\textup{as}}\,\, r\rightarrow 1^{-}. 
 \end{equation}
Also, note that $r^{\Theta}\to 1$ as $r\to 1^{-}$. Then (\ref{l-2}) follows from (\ref{l-eq-5}) and (\ref{l-eq-10}). The lemma is proved.
 \end{proof} 
We need the following results proved recently in \cite{Jin-2}.
\begin{lemma}\label{rec-1}Let $\beta>\alpha>-1$. If $g$ is a multiplier from ${\bf A}_{\alpha}^2$ to ${\bf A}_{\beta}^2$, let $g_a(z)=g(az), z\in \Delta, a\in \Delta$, then for any $a\in \Delta$, 
$$\|g_a\|_{\mathbf{M}_{\alpha, \beta}}\leq \|g\|_{\mathbf{M}_{\alpha, \beta}}.$$
\end{lemma}

\begin{lemma}\label{rec-2}Let $\beta>\alpha>-1$. Define $g_0(z):=(1-z^2)^{-\frac{\beta-\alpha}{2}}, z\in \Delta$. Then we have 
$$\|g_0\|_{\mathbf{M}_{\alpha, \beta}}=\frac{1}{2^{\beta-\alpha}}\frac{\Gamma(\beta+2)}{\Gamma(\alpha+2)}\Big[\frac{\Gamma(1+\alpha/2)}{\Gamma(1+\beta/2)}\Big]^2.$$
\end{lemma}

Now, we start to present the proof of Theorem \ref{th-m}. 
\begin{proof}[Proof of Theorem \ref{th-m}]By symmetry, we may assume that $p\geq q$. We first show that 
$$\|S_{\kappa}^{[p,q]}\|_{{\bf {M}}_{\alpha, \alpha+2p+2q}}^2\leq {\mathcal{N}}_{\alpha, p, q}.$$
Note that all the coefficients of $T_q$ are positive, then we see that for any $z\in \Delta$, 
\[|S_{\kappa}^{[p,q]}(z)|=|p!\frac{z^{p-q}}{(1-z^2)^{p+q}}T_q(z^2)|\leq \frac{p!}{|1-z^2|^{p+q}}T_q(1).\]
On the other hand, the use of the Chu-Vandermonde identity yields 
\begin{eqnarray}T_q(1)&=&\frac{(p-1)!}{(p-q)!}\sum_{d=0}^{q-1}\frac{(-q)_d(1-q)_d}{d!(p-q+1)_d}\nonumber \\
&=&\frac{(p-1)!}{(p-q)!}\frac{(p+1)_{q-1}}{(p-q+1)_{q-1}}=\frac{(p+q-1)!}{p!}.\nonumber\end{eqnarray}
It follows that
\begin{equation}\label{key-ineq}|S_{\kappa}^{[p,q]}(z)|\leq \frac{(p+q-1)!}{|1-z^2|^{p+q}}.\end{equation}
By Lemma \ref{rec-2}, we obtain that
\begin{eqnarray}\lefteqn{\|S_{\kappa}^{[p,q]}\|_{{\bf {M}}_{\alpha, \alpha+2p+2q}^2}}\nonumber \\
&&\leq \frac{(p+q-1)!}{2^{2p+2q}}\frac{\Gamma(\alpha+2p+2q+2)}{\Gamma(\alpha+2)}\Big[\frac{\Gamma(1+\alpha/2)}{\Gamma(1+\alpha/2+p+q)}\Big]^2={\mathcal{N}}_{\alpha, p, q}.\nonumber \end{eqnarray}
Next, we will show that  $$\|S_{\kappa}^{[p,q]}\|_{{\bf {M}}_{\alpha, \alpha+2p+2q}}^2\geq {\mathcal{N}}_{\alpha, p, q}.$$
Let 
\[S_{\kappa}^{[p,q]}(z)=-p!z^{p-q}\frac{T_q(z^2)}{(1-z^2)^{p+q}}:=-p!z^{p-q}W(z^2).\]
Then, since the degree of the polynomial $T_q$ is $q-1$, we can write 
\[W(t):=\sum_{j=0}^{p+q-1}\frac{A_j}{(1-t)^{p+q-j}}.\]
Here, $A_j$ are real numbers and  
\[A_0=\lim\limits_{t\to 1}(1-t)^{p+q}W(t)=T_q(1)=\frac{(p+q-1)!}{p!}.\]
Consequently, we get that
\[S_{\kappa}^{[p,q]}(z)=-\Big[\frac{(p+q-1)!z^{p-q}}{(1-z^2)^{p+q}}+\sum_{j=1}^{p+q-1}\frac{p!A_jz^{p-q}}{(1-z^2)^{p+q-j}}\Big].\]
By using the inequality$$\Big|\sum_{m=1}^{n}a_m\Big|^2\geq |a_1|^2-2|a_1|\sum_{m=2}^{n}|a_m|,$$
we obtain that
\begin{eqnarray}|S_{\kappa}^{[p,q]}(z)|^2&\geq&\Big|\frac{(p+q-1)!z^{p-q}}{(1-z^2)^{p+q}}\Big|^2-2(p+q-1)!\Big|\sum_{j=1}^{p+q-1}\Big|\frac{p!A_jz^{2p-2q}}{(1-z^2)^{2p+2q-j}}\Big|\nonumber \\
&\geq& \Big|\frac{(p+q-1)!z^{p-q}}{(1-z^2)^{p+q}}\Big|^2-\mathbf{M}\sum_{j=1}^{p+q-1}\frac{1}{|(1-z^2)^{2p+2q-j}|}.\nonumber
\end{eqnarray}
Here, we can take $\mathbf{M}=2(p+q-1)!p!\max\limits_{1\leq j\leq p+q-1}\{|A_j|\}$.
Let $\alpha+2<2\lambda<\alpha+3$, from Lemma \ref{rec-1}, we know that
\begin{eqnarray}\label{p-1}\|S_{\kappa}^{[p,q]}\|_{\mathbf{M}_{\alpha, \alpha+2p+2q}}^2&\geq&\sup_{r\in(0,1)}\frac{\|S_{\kappa}^{[p,q]}(\sqrt{r}z)(1-rz^2)^{-\lambda}\|_{\alpha+2p+2q}^{2}}{\|(1-rz^2)^{-\lambda}\|_{\alpha}^2}. 
\end{eqnarray}
On the other hand, we have 
\begin{eqnarray}\label{p-2}\lefteqn{\|S_{\kappa}^{[p,q]}(\sqrt{r}z)(1-rz^2)^{-\lambda}\|_{\alpha+2p+2q}^{2}}
\nonumber \\
&&\geq[(p+q-1)!]^2\|\sqrt{r}z^{p-q}(1-rz^2)^{-p-q-\lambda}\|_{\alpha+2p+2q}^{2}\nonumber \\
&&\quad-\mathbf{M}\sum_{j=1}^{p+q-1}\|(1-rz^2)^{-p-q+\frac{j}{2}-\lambda}\|_{\alpha+2p+2q}^2. 
\end{eqnarray}
From (\ref{a-e-1}) in Lemma \ref{l-2}, we have
\begin{eqnarray}\label{p-3}
\lefteqn{\|(1-rz^2)^{-\lambda}\|_{\alpha}^{2}}
\nonumber \\ &&=\frac{\Gamma(\alpha+2)\Gamma(2\lambda-\alpha-2)}{2^{\alpha+1}[\Gamma(\lambda)]^2}\cdot\frac{1+{o}(1)}{(1-r^2)^{2\lambda-\alpha-2}},\, {\textup{as}}\,\, r\rightarrow 1^{-},  
\end{eqnarray} by taking $\Theta=0$, and 
\begin{eqnarray}\label{p-4}
\lefteqn{\|\sqrt{r}z^{p-q}(1-rz^2)^{-p-q-\lambda}\|_{\alpha+2p+2q}^{2}}
\nonumber \\ &&=\frac{\Gamma(\alpha+2p+2q+2)\Gamma(2\lambda-\alpha-2)}{2^{\alpha+2p+2q+1}[\Gamma(\lambda+p+q)]^2}\cdot\frac{1+{o}(1)}{(1-r^2)^{2\lambda-\alpha-2}},\, {\textup{as}}\,\, r\rightarrow 1^{-}, 
\end{eqnarray} since $2(p+q+\lambda)>\alpha+2p+2q+2$. 

In view of (\ref{a-e-3}), we see that there is a constant $\mathbf{N}>0$ such that 
\begin{equation}\label{p-5}\|(1-rz^2)^{-p-q+\frac{j}{2}-\lambda}\|_{\alpha+2p+2q}^2\leq \mathbf{N},\end{equation}
because $0<2(p+q-\frac{j}{2}+\lambda)<\alpha+2p+2q+2$, for all $1\leq j\leq p+q-1$. 

Combining (\ref{p-1})-(\ref{p-5}), we obtain that
\begin{eqnarray} \lefteqn{\|S_{\kappa}^{[p,q]}\|_{\mathbf{M}_{\alpha, \alpha+2p+2q}}^2} \nonumber \\
&&\geq\sup_{r\in(0,1)}[(p+q-1)!]^2\frac{\Gamma(\alpha+2p+2q+2)}{2^{\alpha+2p+2q+1}[\Gamma(\lambda+p+q)]^2}\frac{2^{\alpha+1}[\Gamma(\lambda)]^2}{\Gamma(\alpha+2)}[1+o(1)]\nonumber \\
&&\qquad -(p+q-1)\mathbf{M}\mathbf{N}(1-r^2)^{2\lambda-\alpha-2}[1+o(1)]. \nonumber 
\end{eqnarray}
Hence, let $r\to 1^{-}$, we have 
\[\|S_{\kappa}^{[p,q]}\|_{\mathbf{M}_{\alpha, \alpha+2p+2q}}^2 \geq[(p+q-1)!]^2 \frac{\Gamma(\alpha+2p+2q+2)}{2^{2p+2q}[\Gamma(\lambda+p+q)]^2}\frac{[\Gamma(\lambda)]^2}{\Gamma(\alpha+2)}.\]
Consequently, let $\lambda \to (\frac{\alpha}{2}+1)^{+}$, we get that

\noindent  \(\|S_{\kappa}^{[p,q]}\|_{\mathbf{M}_{\alpha, \alpha+2p+2q}}^2\)
 \[\geq \frac{[(p+q-1)!]^2}{2^{2p+2q}}\frac{\Gamma(\alpha+2+2p+2q)}{\Gamma(\alpha+2)}\Big[\frac{\Gamma(\alpha/2+1)}{\Gamma(\alpha/2+1+p+q)}\Big]^2=\mathcal{N}_{\alpha,p,q}.\]
This finishes the proof of Theorem \ref{th-m}.\end{proof}

\section{{\bf Final remarks}}
In this section, we first point out that we can reprove the main result in \cite{Do}, Theorem \ref{th-d}, by the results obtained in the present paper. Then we will show that the Koebe function is still the extremal function for certain higher-order Schwarzians. 
\begin{remark}As a key step to prove Theorem \ref{th-d}, after studying the following function 
\[P_q(w)=(1-w)^{2q}\sum_{n=q}^{\infty}\frac{n^2(n-1)\cdots(n-q+1)^2}{n}w^{n-q},\, w\in \Delta, q\in \mathbb{N},\]
Donaire has shown that
\begin{lemma}\label{d-l}
Let $\alpha>-1, p,q\in \mathbb{N}$. Let $f\in \mathcal{S}$. Then, for any fixed $w\in \Delta$, we have
the function $G_f^{[p,q]}(z,w)$ belongs to $\mathbf{A}_{2p}^2$ and  
\begin{equation}\label{d-ineq}\|G_f^{[p,q]}(\cdot,w)\|_{2p}^2\leq \frac{(2p-1)!(2q-1)!}{(1-|w|^2)^{2q}}.\end{equation}
\end{lemma} 
We will use a different method to prove Lemma \ref{d-l}. Let the series expansion of the function $F$ be as follows. 
\[F(z,w)=\sum_{n=1,k=1}\gamma_{n,k}z^kw^n.\]
Here, $\gamma_{n,k}$ are the Grunsky coefficients of $f$. 
We know that 
\[G_{p,q}(z,w)=\sum_{k=p}^{\infty}\frac{k!}{(k-p)!}\Big(\sum_{n=q}^{\infty}\gamma_{n,k}\frac{n!}{(n-q)!}w^{n-q}\Big)z^{k-p}\]
so that, as is done in \cite[Pages 320-321]{Do}, by using the Grunsky inequality,  
\begin{eqnarray}
\|G_{p,q}(\cdot, w)\|_{2p}^2&\leq& (2p-1)!\sum_{n=q}^{\infty}\frac{n!}{n(n-q)!}w^{2n-2q}\nonumber \\
&=& (2p-1)!\sum_{n=q}^{\infty}\frac{n^2(n-1)^2\cdot(n-q+1)^2}{n}|w|^{2n-2q}.\nonumber
\end{eqnarray} 
Then, by a result about $P_q$, Donaire proved Lemma \ref{d-l}. 

Here, we note that, for $t\in [0,1)$, 
\[\sum_{n=q}^{\infty}\frac{n^2(n-1)\cdots(n-q+1)^2}{n}t^{2n-2q}=-\frac{\partial^{2q}K}{\partial z^q \partial w^q}|_{z=w=t}\]
\[=-\frac{\partial^{2q}}{\partial z^q \partial w^q}\log(1-zw)|_{z=w=t}=-S_{\kappa}^{[q,q]}(t).\]
It follows from (\ref{key-ineq}) that 
\[\sum_{n=q}^{\infty}\frac{n^2(n-1)\cdots(n-q+1)^2}{n}|w|^{2n-2q}\leq \frac{(2q-1)!}{(1-|w|^2)^{2q}}.\]
This proves Lemma \ref{d-l} by a different way. Then we can finish the proof of Theorem \ref{th-d} after repeating the arguments of the proof of Theorem $1$ in \cite[Page 321]{Do}.
We have given an alternative proof of Theorem \ref{th-d}. As a byproduct, we further observe that  
$$P_q(w^2)=-(1-w^2)^{2q}S_{\kappa}^{[q,q]}(w),\,w\in \Delta.$$
\end{remark}

\begin{remark}
We know from Proposition \ref{th-1} and Theorem \ref{th-d} that $S_f^{[p,q]}$ belongs to $\mathcal{B}_{p+q}$ for all $f\in \mathcal{S}$. Moreover, we have
\begin{proposition}\label{last-p}
Let $f\in \mathcal{S}$. Then we have
\begin{equation}\label{jia-m}
\|S_{f}^{[p,q]}\|_{\mathcal{B}_{p+q}}=\sup\limits_{z\in \Delta}|S_{f}^{[p,q]}(z)|(1-|z|^2)^{p+q}\leq \sqrt{(2p-1)!(2q-1)!}\end{equation}
\end{proposition}
\begin{proof}
First, note that 
\begin{equation}\label{jia-eq-1}S_{f}^{[p,q]}(z)=G_{p,q}(z,z)=\sum_{k=p}^{\infty}\sum_{n=q}^{\infty}\gamma_{n,k}\frac{k!}{(k-p)!}\frac{n!}{(n-q)!}z^{k+n-p-q}.\end{equation}
Recall that the following Grunsky inequality 
\[\Big|\sum_{k=p}^{\infty}\sum_{n=q}^{\infty}\gamma_{n,k}x_k x_n\Big|^2\leq \sum_{k=p}^{\infty}\frac{|x_k|^2}{k}\sum_{n=q}^{\infty}\frac{|x_n|^2}{n}\]
holds if the two series both converges in the right side of the inequality, see \cite[Pages 60-61]{Po}. Then, let $x_n=\frac{n!}{(n-p)!}z^{n-q}$ in (\ref{jia-eq-1}), we obtain that
\[
|S_{f}^{[p,q]}(z)|^2\leq \sum_{k=p}^{\infty}\frac{[k!]^2}{k[(k-p)!]^2}|z|^{2(k-p)}\sum_{n=q}^{\infty}\frac{[n!]^2}{n[(n-q)!]^2}|z|^{2(n-q)}.\]
On the other hand, as above, we have 
\[
\sum_{m=j}^{\infty}\frac{[m!]^2}{m[(m-j)!]^2}|z|^{2(m-j)}=|S_{\kappa}^{[j,j]}(|z|)|\leq \frac{(2j-1)!}{(1-|z|^2)^{2j}},\ j=p,q. 
\]
It follows that 
\[
|S_{f}^{[p,q]}(z)|^2\leq \frac{(2p-1)!(2q-1)!}{(1-|z|^2)^{2p+2q}}.\]
That is
\[|S_{f}^{[p,q]}(z)|(1-|z|^2)^{p+q}\leq \sqrt{(2p-1)!(2q-1)!}.\]
This finish the proof of Proposition \ref{last-p}.
\end{proof}
\end{remark}
 
\begin{remark}Recalling the arguments about $S_{\kappa}^{[p,q]}$ in the proof of Theorem \ref{th-m}, we see that \[\|S_{\kappa}^{[p,q]}\|_{\mathcal{B}_{p+q}}=(p+q-1)!.\]
This means that the inequality (\ref{jia-m}) is sharp for the Koebe function when $p=q$.
That is to say, the Koebe function is still the extremal function for certain higher-order Schwarzians. 

Also, notice that when $p>q$, i.e., $p\geq q+1$,
\[\frac{(2p)!}{(p+q)!}=(p+p)\cdot (p+p-1)\cdots (p+q+1)\]
\[\qquad\qquad >(q+p)(q+p-1)\cdots(q+q+1)=\frac{(p+q)!}{(2q)!},\]
and $2p\cdot2q<(p+q)^2$. Hence we obtain that 
\[(2p-1)!(2q-1)!=\frac{(2p)!(2q)!}{2p\cdot2q}>\frac{[(p+q)!]^2}{(p+q)^2}=[(p+q-1)!]^2,\]
for $p>q$. It follows that
\[\sqrt{(2p-1)!(2q-1)!}>(p+q-1)!,\]for $p\neq q$. Then, it is natural to ask 
\begin{question}
Can the upper bound of the inequality (\ref{jia-m}) be sharpened to $(p+q-1)!$ when $p\neq q$?
\end{question}
 \end{remark}
 
\begin{remark}
In Lemma \ref{d-l}, when $p=q=1$, a direct computation yields that 
\[G_{\kappa}^{[1,1]}(z,w)=\frac{1}{(1-zw)^2},\]
so that
\[
\| G_{\kappa}^{[1,1]}(\cdot, w) \|_{2}^2 = \frac{1}{(1 - |w|^2)^2}.
\]
This means that the Koebe function attains equality in inequality (\ref{d-ineq}). However, when $p\geq 2, q\in \mathbb{N}$, the inequality (\ref{d-ineq}) in Lemma \ref{d-l} is strict for any univalent function \(f \in \mathcal{S}\). 
Indeed, from the proof of Lemma \ref{d-l} in \cite{Do}, we obtain that
\[
\| G_f^{[p,q]}(\cdot, w) \|_{2p}^2 = (2p-1)! \sum_{k=p}^{\infty} 
\frac{k(k-1)\cdots(k-p+1)}{(k+p-1)\cdots(k+1)}\left| \sum_{n=q}^{\infty} \gamma_{n,k} \frac{n!}{(n-q)!} w^{n-q} \right|^2.
\]
Note that, for $p\geq 2$, we have
\[
\frac{(k-1)\cdots(k-p+1)}{(k+p-1)\cdots(k+1)}<1,
\]
for each $k\geq p$ so that
\[
\| G_f^{[p,q]}(\cdot, w) \|_{2p}^2<(2p-1)! \sum_{k=p}^{\infty}k
\left| \sum_{n=q}^{\infty} \gamma_{n,k} \frac{n!}{(n-q)!} w^{n-q} \right|^2.
\]
Consequently, for $p\geq 2, q\in \mathbb{N}$, we get that 
\[
\| G_f^{[p,q]}(\cdot, w) \|_{2p}^2 < \frac{(2p-1)!(2q-1)!}{(1 - |w|^2)^{2q}},
\]
for any $f\in \mathcal{S}$ even when \(f\) is the Koebe function. That is to say, for $p\geq 2, q\in \mathbb{N}$, the upper bound of (\ref{d-ineq}) can not be attainable and there is no extremal 
function in Lemma \ref{d-l}. A more refined question is 
\begin{question}Is it true that for every $f\in \mathcal{S}$,
\[\sup_{w\in \Delta}(1-|w|^2)^{2q}\| G_f^{[p,q]}(\cdot, w) \|_{2p}^2<(2p-1)!(2q-1)!, 
\]
when $p\geq 2, q\in \mathbb{N}$? 
\end{question}
\end{remark}

\section{{\bf Acknowledgements}}
The author was supported by National Natural Science Foundation of China (Grant No. 11501157). 

\begin{spacing}{1.2}

\end{spacing}
\end{document}